\documentclass[12pt]{amsart}
\usepackage{amsfonts,amssymb,amscd,amsmath,enumerate,verbatim,calc}
\usepackage[colorlinks,linkcolor=blue,pagebackref=true, citecolor=red]{hyperref}
\usepackage[all]{xy}
\usepackage{multicol}
\usepackage[mathscr]{euscript} 
\usepackage{mathrsfs}
\usepackage[left=2cm,right=2cm,top=2.5cm,bottom=2.5cm]{geometry}

\usepackage{graphicx}
\usepackage{color}
\usepackage{tikz}
\usepackage{tikz-cd}

\usepackage{vwcol}

\newcommand{\m}{\mathfrak{m} }

\newcommand{\f}{\mathcal{F}}
\newcommand{\R}{\mathcal{R}}

\newcommand{\Z}{\mathbb{Z} }
\newcommand{\N}{\mathbb{N} }

\newcommand{\C}{\mathbb{C} }

\newcommand{\Ass}{\operatorname{Ass}}

\newcommand{\grade}{\operatorname{grade}}

\newcommand{\vol}{\operatorname{vol}}

\newcommand{\sym}{\operatorname{Sym}}

\theoremstyle{plain}

\newtheorem{theorem}{Theorem}[section]
\newtheorem{corollary}[theorem]{Corollary}

\theoremstyle{definition}

\newtheorem{remark}[theorem]{Remark}

\newtheorem{example}[theorem]{Example}
\newtheorem*{example*}{\it Example}

\newtheorem*{acknowledgements}{Acknowledgements}
\theoremstyle{remark}
\newtheorem*{claim*}{\it Claim}

\newtheorem*{case*}{\it Case}
\newtheorem*{note*}{\it Note}

\title[Algorithms for computing mixed multiplicities]{ALGORITHMS FOR COMPUTING MIXED MULTIPLICITIES, \\ [3mm] MIXED VOLUMES AND SECTIONAL MILNOR NUMBERS}
\author{Kriti Goel}
\address{Department of Mathematics, University of Utah, Salt Lake City, UT 84112, USA}
\email{kritigoel.maths@gmail.com}

\author{Vivek Mukundan}
\address{Department of Mathematics, Indian Institute of Technology Delhi, New Delhi, 110016, India}
\email{vmukunda@iitd.ac.in}

\author{Sudeshna Roy}
\address{Department of Mathematics, Chennai Mathematical Institute, Siruseri, Kelambakkam 603103, India}
\email{sudeshnaroy.11@gmail.com}

\author{J. K. Verma}
\address{Department of Mathematics, Indian Institute of Technology Bombay, Mumbai, 400076, India}
\email{jkv@math.iitb.ac.in}

\begin{document}

\begin{abstract}
We present a package \texttt{MixedMultiplicity} for computing mixed multiplicities of ideals in a Noetherian ring which is either local or a standard graded algebra over a ﬁeld. This enables us to find mixed volumes of convex lattice  polytopes and sectional Milnor numbers of hypersurfaces with an isolated singularity. The algorithms make use of the defining equations of the multi-Rees algebra of ideals, which are obtained by generalising a result of Cox, Lin and Sosa \cite{CLS}.
\end{abstract}

\maketitle

\section{Introduction}

The objective of this article is to describe the Macaulay2 package \texttt{MixedMultiplicity} which computes the mixed multiplicities of ideals having positive grade in a Noetherian ring, mixed volume of a collection of convex lattice polytopes and sectional Milnor numbers of hypersurfaces with an isolated singularity. One of the main steps of the algorithms is the computation of the defining equations of the multi-Rees algebra 
\[\R(I_1,\ldots,I_s)=R[I_1t_1, \ldots, I_st_s]=\bigoplus\limits_{a_1, \ldots, a_s \geq 0} I_1^{a_1} \cdots I_s ^{a_s}t_1^{a_1} \cdots t_s^{a_s} \subseteq R[t_1, \ldots, t_s]\] 
of ideals $I_1,\ldots,I_s$ in a Noetherian ring $R$. When $R$ is a polynomial ring over any field $k$ and $I_1,\ldots, I_s$ are monomial ideals in $R$, D. Cox, K.-N. Lin, and G. Sosa give an explicit formula for the defining equations of $\R(I_1,\ldots,I_s)$ in \cite{CLS}. In this article, we obtain, in Theorem \ref{TisRees}, an analogue of their result for any set of ideals $I_1, \ldots, I_s$ in a Noetherian ring $R$ such that each $I_i$ has positive grade. The latter condition is always satisfied when $R$ is a domain or for any ideal of positive height in a reduced ring or in a Cohen-Macaulay ring. Using this result, we write a function \texttt{multiReesIdeal} to compute the defining ideal of a multi-Rees algebra in Macaulay2. It should be noted that the command \texttt{reesIdeal} in the Macaulay2 package \texttt{ReesAlgebra} \cite{Eisenbud} is already available to compute the defining ideal of the Rees algebra of a module \cite{Eisenbud-Huneke-Ulrich}, in particular, the defining ideal of a multi-Rees algebra. But, the algorithm presented in this article runs faster in many cases. When the ring is not a domain, the algorithm follows an analogue of the \texttt{reesIdeal}, albeit for the multiple ideal setting.

%The function \texttt{multiReesIdeal} also tackles the non-domain case by computing the defining ideal of the multi-Rees algebra as the kernel of the map between two symmetric algebras. 

% Several authors have proposed algorithms to determine the defining equations of the multi-Rees algebra. 
Several authors have proposed algorithms to determine the defining equations of the multi-Rees algebra for specific classes of ideals. For example,  J. Ribbe \cite[Proposition 3.1, Proposition 3.4]{Ribbe}, K.-N. Lin and C. Polini \cite[Theorem 2.4]{Lin-Polini}, G. Sosa \cite[Lemma 2.1]{Sosa},  and  B. Jabarnejad \cite[Theorem 1]{Jabarnejad}. 
% They  assumed $I_i=\m^{a_i}$ or $I_i=I^{a_i}$, where $I$ is an ideal of linear type, or $I_i$'s are monomial ideals. 
The algorithm for computing defining equations of $\R(I_1, \ldots, I_s)$ helps us to construct algorithms to compute \emph{mixed multiplicities} (see Section \ref{SecMixedMul}), \emph{mixed volume} (see Section \ref{SecMixedVol}) and \emph{sectional Milnor numbers} (see Section \ref{SecMilNo}) in the general setting. Observe that to compute mixed multiplicities of ideals one can always assume each ideal to have positive grade using a standard trick (see Remark \ref{exp}). 
%In this article, we introduce and give an overview of the Macaulay2 package \texttt{MixedMultiplicity}. 

For any ideal $J$ in a Noetherian ring $R$, we denote the common length of the maximal $R$-sequences in $J$ by $\grade(J)$. Let $I_0,I_1,\ldots,I_r$ be a set of ideals in a Noetherian ring of dimension $d \geq 1$, which is either local or a standard graded algebra over a ﬁeld. Further assume that $I_0$ is primary to the maximal ideal (resp. maximal graded ideal) and $\grade (I_j)>0$ for all $j.$ Let $\underline{a}=(a_0,a_1,\ldots,a_r) \in \N^{r+1}$ with $|\underline{a}|=d-1.$ 
%Here $\N=\{0, 1, 2, \ldots\}$. 
The function \texttt{MixedMultiplicity} computes the mixed multiplicity $e_{\underline{a}}(I_0 \mid I_1,\ldots,I_r).$ Using the results of N. V. Trung - J. K. Verma and B. Teissier, mixed volumes and sectional Milnor numbers can be identified with mixed multiplicities of ideals over polynomial rings. Let $Q_1,\ldots,Q_n$ be a collection of lattice polytopes in $\mathbb{R}^n.$ The function \texttt{mMixedVolume} computes the mixed volume of $Q_1,\ldots,Q_n.$ Let $R=k[x_1,\ldots,x_n]$ be a polynomial ring in $n$ variables, $\m$ be the maximal graded ideal and $f \in R$ be any polynomial. The function \texttt{secMilnorNumbers} computes the sectional Milnor numbers by calculating the mixed multiplicities $e(\m^{[n-i]}, J(f)^{[i]})$, $0 \leq i \leq n-1$, where $J(f) = (f_{x_1},\ldots,f_{x_n})$ is the ideal generated by the partial derivatives of $f.$ Many researchers, including M. Herrmann (\cite{HERRMANN}), J. K. Verma (\cite{VERMA1992}), and C. D'Cruz (\cite{D'Cruz2003}), have expressed the multiplicities of Rees algebras, extended Rees algebras, and certain form rings in terms of mixed multiplicities. Therefore \texttt{MixedMultiplicity} package is also helpful in this regard. For any unexplained invariants and definitions used in this article, the reader may refer to \cite{Bruns-Herzog}, \cite{Eisenbud-Huneke-Ulrich}, and \cite{Huneke-Swanson}.

\begin{acknowledgements}
 We thank the reviewers for their comments to improve this article. The authors thank Wolfram Decker and David Eisenbud for their help and encouragement. The first author thanks D. Grayson and M. Stillman for their comments and suggestions to improve the exposition and fix the grading scheme in the algorithm. The third author is grateful to the Infosys Foundation for providing partial financial support. At the beginning of the project, the first and third author were Ph.D. students at Indian Institute of Technology Bombay, Mumbai, India and were supported by UGC-SRF fellowship, Govt. of India.
\end{acknowledgements}

\section{Defining ideal of multi-Rees algebra of ideals}

An explicit formula for the defining ideal of the multi-Rees algebra of a finite collection of monomial ideals in a polynomial ring was given by D. Cox, K.-N. Lin, and G. Sosa in \cite{CLS}. In this section, we generalize their result to find the defining ideal of the multi-Rees algebra of a collection of positive grade ideals in a Noetherian ring. We use this result to write a Macaulay2 algorithm to compute the defining ideal when the base ring is a domain. We further provide a seperate algorithm for the non-domain case.

Let $R$ be a Noetherian ring and $I_1,\ldots, I_s \subseteq R$ be ideals. Suppose that $I_i = \langle f_{ij} \mid j=1,\ldots,n_i \rangle$ for all $i = 1,\dots,s$. Let $\R(I_1,\ldots,I_s)$ be the multi-Rees algebra of ideals $I_1,\ldots,I_s.$ Consider the set of indeterminates $\underline{Y} = \{Y_{ij} \mid i=1,\ldots,s, j=1,\ldots,n_i \}$ and $\underline{T}=(T_1,\ldots,T_s)$ over $R.$ Define an $R$-algebra homomorphism
$R[\underline{Y}] \overset{\varphi} \longrightarrow \R(I_1,\ldots,I_s) \subseteq R[\underline{T}]$ such that $\varphi(Y_{ij}) = f_{ij}T_i$, for all $i=1,\ldots,s$, $j=1,\ldots,n_i$ and $\varphi(r) = r$ for all $r \in R.$ Then $\mathcal{R}(I_1,\ldots,I_s) \simeq R[\underline{Y}] / \ker(\varphi).$ The ideal $\ker \varphi$ is called the defining ideal of $\R(I_1, \ldots, I_s)$. We give an explicit description of $\ker(\varphi).$ 

\begin{theorem} \label{TisRees}
Let $R$ be a Noetherian ring and $I_1,\ldots, I_s \subseteq R$ be ideals of positive grade. For each $i$, consider some generating set $\{f_{ij} \mid j=1, \ldots, n_i\}$ of $I_i$ which contains at least one nonzerodivisor $f_{ij_i}$. We set $h=\prod_{i=1}^s f_{ij_i}$ and set
\[ \Gamma = \langle Y_{ij} f_{ij'} - Y_{ij'} f_{ij} \mid i=1,\ldots,s \text{ and } j,j' \in \{1,\ldots,n_i\} \rangle : h^\infty \subseteq R[\underline{Y}]. \]	
Then $\Gamma \subseteq R[\underline{Y}]$ is the defining ideal of $\mathcal{R}(I_1,\ldots,I_s)$. 
\end{theorem}

\begin{proof}
	Without loss of generality, we may assume that $j_i=1$ for all $i=1,\ldots,s$ and $h=\prod_{i=1}^s f_{i1}$. Consider the ring homorphism $\phi: R \to R[f_{11}^{-1}, f_{21}^{-1}, \ldots, f_{s1}^{-1}] \cong R[h^{-1}]$. The defining ideal of $\mathcal{R}(\phi(I_1),\ldots,\phi(I_s))$ of the ideals $\phi(I_i)=(1, f_{i2}/f_{i1}, \ldots, f_{in_i}/f_{i1})$ for $i=1, \ldots, s$, is $J:=J_1+\cdots+J_s$ in $R[h^{-1}][\underline{Y}]\cong R[\underline{Y}]_h$, 
	% (here we have used the fact that $f_{11},f_{21},\ldots,f_{s1}$ are nonzero divisors, fails when $h=0$ or some $f_{ij}$ is nilpotent), 
	where $J_i:=(Y_{i2}-\frac{f_{i2}}{f_{i1}}Y_{i1}, \ldots,  Y_{in_i}-\frac{f_{in_i}}{f_{i1}}Y_{i1})$.  We claim that $\phi^{-1}(J) = \Gamma.$ Observe that for all $j \neq j',$ and for all $i$,
	\[f_{ij}Y_{ij'}-f_{ij'}Y_{ij}=f_{ij}\left(Y_{ij'}-\frac{f_{ij'}}{f_{i1}}Y_{i1}\right)-f_{ij'}\left(Y_{ij}-\frac{f_{ij}}{f_{i1}}Y_{i1}\right) \in J_i. \] 
	So $\Gamma \subseteq \phi^{-1}(J)$. Now let $r \in \phi^{-1}(J)$. Then $\phi(r) \in J$, i.e.,
	\[\frac{r}{1} =\sum_{i=1}^s\sum_{j=2}^{n_i} \frac{a_{ij}}{f_{i1}^{m_{ij}}}\left(Y_{ij}-\frac{f_{ij}}{f_{i1}}Y_{i1}\right)\]
	for some $a_{ij} \in R$. Thus we have 
	\[h^{m} r \in (f_{i1}Y_{ij}-f_{ij}Y_{i1}\mid 1 \leq i \leq s \text{ and } 1 \leq j \leq n_i)\subseteq (f_{ij}Y_{ij'}-f_{ij'}Y_{ij} \mid 1 \leq i \leq s \text{ and } 1 \leq j, j' \leq n_i) \]
	for some $m \geq \max\{m_{ij}\mid 1 \leq i \leq s, 1 \leq j \leq n_i\}+1$. Therefore, $r \in \Gamma$ and hence the claim holds. From \cite[Proposition 3.11 $\mbox{(iii)}$]{Atiyah-MacDonald}, we get that $\phi^{-1}(J)$ is the defining ideal of $\R(I_1, \ldots,I_s)$, as $h^t$ is a nonzerodivisor on $R[\underline{Y}]/\Gamma$ for all $t \geq 1$.
\end{proof}

When $R$ is a domain or when a list of nonzerodivisors (one each from the list of ideals with positive grades) is provided by the user, the function \texttt{multiReesIdeal} computes the defining ideal of the multi-Rees algebra using Theorem \ref{TisRees}.  %The latter condition, for example, is true for any ideal of positive height in a reduced ring or in a Cohen-Macaulay ring.

%
%\begin{remark}\label{positive height ideals}
% Notice that Theorem \ref{TisRees} is valid for domains, since any nonzero ideal in it has positive grade. Moreover, any ideal of positive height in a reduced ring or in a Cohen-Macaulay ring has positive grade. So the theorem is also applicable if we take a set of ideals with positive height in such rings. Using Macaulay2, one can easily compute the grade of an ideal. 
%\end{remark}

\smallskip
\noindent
{\bf Algorithm (Part I).} Let $I_1,\ldots,I_s$ be ideals of a Noetherian ring $R$ with $\grade I_i>0$ for all $i$ and let $a_1, \ldots, a_s$ be a set of nonzerodivisors, where $a_i$ belongs to the generating set of $I_i$ for all $i.$ When $R$ is a domain, the function picks $a_i$ to be the first element in the generating set of $I_i$ for each $i$.\\
{\bf Input}: The list $W = \{\{I_1,\ldots,I_s\},\{a_1, \ldots, a_s\}\}$, or $W = \{I_1,\ldots,I_s\}$ if $R$ is a domain.
\begin{enumerate}[1.]
\item Define a polynomial ring $S$ by attaching $m$ indeterminates to the ring $R$, where $m$ is the sum of the number of generators of all the ideals. 

\item For each ideal $I_i$, construct a matrix $M(i)$ whose first row consists of the generators of the ideal and the second row consists of the indeterminates. 

\item Add the $2 \times 2$ minors of these matrices to get an ideal $L.$ 

\item To get the defining ideal, saturate $L$ with the product of $a_i$'s. 
\end{enumerate} 
{\bf Output}: The defining ideal of the Rees algebra $\R(I_1,\ldots,I_s).$ 

\smallskip
The elements of the defining ideal are assigned $\mathbb{N}^{s+1}$ degree by the function, where the first $\mathbb{N}^{s}$ coordinates point to the component of $\mathcal{R}(I_1,\ldots,I_s)$ where the element lies and the last coordinate is the degree of the element. In order to compute the multi-Rees ideal $\R(I,J)$ using the function \texttt{reesIdeal}, one needs to enter $I \oplus J$  and in general, are slower than the \texttt{multiReesIdeal} routine. 

\begin{example}
%Let $R=\mathbb{Q}[x,y,z]$, $I=(x^4+y^2z^2,xy^2z)$, and $J=(y^3+z^3,x^2y+xz^2).$ To calculate the defining ideal of the Rees algebras $\R(I,J)$, we execute the algorithm \texttt{multiReesIdeal} in Macaulay2.
{\small
\begin{multicols}{2}
\begin{verbatim}
i1 : R = QQ[x,y,z];
i2 : I = ideal(x^4+y^2*z^2,x*y^2*z);
i3 : J = ideal(y^3+z^3, x^2*y+x*z^2);
i4 : time multiReesIdeal {I,J};
      -- used 0.0131127 seconds
i5 : transpose gens oo
o5 = {0,-1,-6} | (x2y+xz2)X_2+(-y3-z3)X_3 |
     {-1,0,-8} | xy2zX_0+(-x4-y2z2)X_1    |
     
i6 : (first entries gens o4)/degree 
o6 : {{0, 1, 6}, {1, 0, 8}}  
i7 : time multiReesIdeal({I,J}, {I_1, J_0});
     -- used 0.0629737 seconds  
i8 : time reesIdeal directSum(module I, module J);
     -- used 0.40043 seconds
\end{verbatim}
%\begin{verbatim}
%	i1 : R = QQ[x,y,z];
%	i2 : I = ideal(x^4+y^2*z^2,x*y^2*z);
%	i3 : J = ideal(y^3+z^3, x^2*y+x*z^2);
%	i4 : time multiReesIdeal({I,J}, {I_1, J_0});
%	-- used 0.0629737 seconds
%	i5 : time multiReesIdeal({I,J}, {I_0, J_0});
%	-- used 0.0130501 seconds
%	i6 : time multiReesIdeal {I,J};
%	-- used 0.0131127 seconds
%	i7 : (first entries gens oo)/degree 
%	o7 : {{0, 1, 6}, {1, 0, 8}}    
%\end{verbatim}
\end{multicols}
}
\end{example} 
In the following example, our algorithm works faster than the function \texttt{reesIdeal}.
\begin{example}
%Let $R=(\mathbb{Z}/32003\mathbb{Z})[y_0,y_1,y_2,y_3]$ and $C$ be an ideal defined parametrically by $(t^3,t^5,t^7, t^{12}).$ We make the following session in Macaulay2.
{\small
\begin{multicols}{2}	
\begin{verbatim}
i1 : ZZ/32003[y_0..y_4];
i2 : C = trim monomialCurveIdeal(R, {3, 5, 7, 12});
i3 : time multiReesIdeal (C, C_0);
          -- used 167.118 seconds
          i4 : time reesIdeal (C, C_0);
          -- used 295.675 seconds
\end{verbatim}
\end{multicols}
}
\end{example}

%\begin{remark}
%The routine also accepts a list of nonzerodivisors in each ideal $I_i,1\leq i\leq s$ as an optional argument. In light of Remark \ref{positive height ideals}, this enables us 
%% to use Remark \ref{positive height ideals} 
%to extend the routine to accept ideals with positive height in a reduced Noetherian ring. 
%% This enables us to compute the element with respect to which we saturate. 
%\end{remark}

\subsection{Routine for the non-domain case} In this section we present an algorithm to find the defining ideal of the Rees algebra using the definitions of Rees algebra.
This method does not have any requirements on the grade of the ideals or the domain-ness of the ring, but it seems to be comparably slower than the previous method. 

We can construct the Rees algebra of $I_i$ as the kernel of the map $\varphi_i:R[Y_{i1},\dots,Y_{in_i}]\rightarrow R[T_i]$ where $\varphi_i(Y_{ij})=f_{ij}T_i$ for $j=1, \ldots, n_i$. Notice that the $(\ker\varphi_i )R[\underline{Y}]\subseteq \ker \varphi$. Suppose that $\phi_i$ is the presentation matrix of $I_i$. Then the symmetric algebra $\sym(I_i)$ has a presentation $R[Y_{i1},\dots,Y_{in_i}]/\mathcal{L}_i$ where $\mathcal{L}_i=I_1([Y_{i1},\dots,Y_{in_i}]\cdot \varphi_i)$. Clearly, $\mathcal{L}_i \subseteq \ker \varphi_i\subset \ker\varphi$. So the map $\varphi_i$ factors through the symmetric algebra $\sym(I_i)$. Now $\sym(I_1)\otimes\cdots\otimes \sym(I_s)$ has the presentation $R[\underline{Y}]/(\mathcal{L}_1+\cdots+\mathcal{L}_s)$. Since each of $\mathcal{L}_i\subseteq\ker \varphi$, the map $\varphi$ also factors through $\sym(I_1)\otimes\cdots\otimes \sym(I_s)$. Thus to find the defining ideal of the multi Rees algebra $\mathcal{R}(I_1,\dots,I_s)$ it is enough to find the kernel of the surjective map $\sym(I_1)\otimes\cdots\otimes \sym(I_s)\rightarrow \mathcal{R}(I_1,\dots,I_s)$.

\vspace{0.15cm}
\noindent
{\bf Algorithm (Part II).} Let $I_1,\ldots,I_s$ be ideals in the Noertherian ring $R$. \\
{\bf Input}: The list $W = \{I_1,\ldots,I_s\}.$
\begin{enumerate}[1.]
\item For each ideal $I_i$ compute the presentation $F_i'\xrightarrow{\phi_i}F_i \rightarrow R_i\rightarrow 0$ where the entries of $\varphi_i$ are the generating set for $I_i$. 
\item Now compute the source symmetric algebra  $\sym(F_1')\otimes\cdots\otimes\sym(F_s')$ and the target symmetric algebra $\sym(F_1)\otimes\cdots\otimes\sym(F_s)$ of the map $\phi_1 \otimes \cdots \otimes \phi_s$.
\item Compute the map between the symmetric algebra of the source and target and return kernel of the above map.
\end{enumerate} 
{\bf Output}: The defining ideal of the Rees algebra $\R(I_1,\ldots,I_s).$ 

In the following example the ring $U$ is not a domain and hence the algorithm uses the above method. As expected, the computational time in the case a nonzero divisor is given as an optional input is faster than the case where no optional input is given.
%We also compare the computational time with the case when a nonzerodivisor is also provided as an input. The latter computation uses the first algorithm and also turns out to be faster in the following  example.
\begin{example}
%	Let $U = \QQ[a,b,c]/I_2(m)$, where $m = \left[ \begin{array}{ccc} a & b & c \\ b & c & a \end{array} \right].$ 
	{\small 
	\begin{multicols}{2}
	\begin{verbatim}
	i1 : T = QQ[a,b,c];
	i2 : m = matrix{{a,b,c},{b,c,a}};
	i3 : U = T/minors(2,m);
	i4 : J = ideal vars U;
	i5 : time multiReesIdeal J;
	-- used 0.0977545 seconds
	i6 : time multiReesIdeal (J, a);
	-- used 0.0142101 seconds
	\end{verbatim}
%	\begin{verbatim}
%	i1 : T = QQ[a,b,c];
%	i2 : m = matrix{{a,b,c},{b,c,a}};
%	i3 : U = T/minors(2,m);
%	i4 : J = ideal vars U;
%	i5 : time multiReesIdeal J;
%	-- used 0.0977545 seconds
%	i6 : time multiReesIdeal (J, a);
%	-- used 0.0142101 seconds
%	i7 : time multiReesIdeal ({J,J}, {a,a});
%	-- used 0.01686 seconds
%	i8 : time reesIdeal directSum(module J, module J);
%	-- used 0.0842076 seconds
%	\end{verbatim}
	\end{multicols}
	}  
\end{example}

\section{Computation of mixed multiplicities of ideals}\label{SecMixedMul}

Let $I_1,\ldots, I_r$ be ideals of positive height in a local ring $(A,\m)$ {\rm(}or a standard graded algebra over a field and $\m$ is the maximal graded ideal{\rm)} and let $I_0$ be an $\m$-primary ideal. In \cite{Verma-katz}, the authors prove that $\ell\left(I_0^{u_0}I_1^{u_1}\cdots I_r^{u_r}/I_0^{u_0+1}I_1^{u_1}\cdots I_r^{u_r}\right)$ is a polynomial $P(\underline{u})$, for $u_i$ large, where $\underline{u}=(u_0, \ldots, u_r)$. Write this polynomial in the form 
\[P(\underline{u})=\sum_{\substack{\alpha \in \N^{r+1}\\ |\alpha|=t}} \frac{1}{\alpha!}e_{\alpha}(I_0 \mid I_1,\ldots,I_r) u^{\alpha}+\text{ lower degree terms, }\] 
where $t = \deg P(\underline{u}), \alpha=(\alpha_0, \ldots, \alpha_r) \in \N^{r+1}, \alpha! = \prod_{i=0}^{r} \alpha_i!$ and $|\alpha| = \sum_{i=0}^{r} \alpha_i.$ If $|\alpha|=t$, then $e_{\alpha}(I_0 \mid I_1,\ldots,I_r)$ are called the {\it mixed multiplicities} of the ideals $I_0, I_1,\ldots,I_r$. In \cite{Trung-Verma}, the authors prove the following result.

\begin{theorem}[{\rm{\cite[Theorem 1.2]{Trung-Verma}}}]
	Set $\displaystyle R = R(I_0\mid I_1, \ldots, I_r) = \bigoplus_{(u_0,u_1, \ldots, u_r) \in \N^{r+1}} \frac{I_0^{u_0}I_1^{u_1}\cdots I_r^{u_r}}{I_0^{u_0+1}I_1^{u_1}\cdots I_r^{u_r}}.$ Assume that $d = \dim A/(0 :I^{\infty})\geq 1$, where $I=I_1\cdots I_r$. Then $\deg P_R(\underline{u})=d-1$, where $P_R(\underline{u})$ is the Hilbert polynomial of $R$.
\end{theorem}

In \cite{Verma-Katz-Mondal}, D. Katz, S. Mandal, and J. K. Verma, gave a precise formula for the Hilbert polynomial of the quotient of a bi-graded algebra over an Artinian local ring. This result can be generalized to  the case of quotient of a multi-graded algebra over an Artinian local ring and we skip the proof as the technique is similar. Let $S$ be an Artinian local ring and $A = S[X_1,\ldots,X_r]$ be an $\N^r$-graded ring over $S,$ where for $1 \leq i \leq r$, $X_i=\{X_i(0), \ldots, X_i(s_i)\}$ is a set of indeterminates. Set $\underline{u}=(u_1, \ldots, u_r) \in \N^r$ and $|u|=u_1+\cdots+u_r$. Then $A = \bigoplus_{\underline{u}\in \N^r} A_{\underline{u}}$, where $A_{\underline{u}}$ is the $S$-module generated by monomials of the form $P_1\cdots P_r$, where $P_i$ is a monomial of degree $u_i$ in $X_i$. An element in $A_{\underline{u}}$ is called multi-homogeneous of degree $\underline{u}$. An ideal $I \subseteq A$ generated by multi-homogeneous elements is called a multi-homogeneous ideal. Then $R=A/I$ is an $\N^r$-graded algebra with $\underline{u}$-graded component $R_{\underline{u}}=A_{\underline{u}}/I_{\underline{u}}$. The Hilbert function of $R$ is defined as $H(\underline{u})=\lambda(R_{\underline{u}}),$ where $\lambda$ denotes the length as an $S$-module. Set $\underline{t}^{\underline{u}}=t_1^{u_1}\cdots t_r^{u_r}$. The Hilbert series of $R$ is given by $HS(R,\underline{t})=\sum_{\underline{u}\in \N^r} \lambda(R_{\underline{u}})\underline{t}^{\underline{u}}$. Then there exists a polynomial $N(t_1, \ldots, t_r) \in \Z[t_1, \ldots, t_r]$ so that $HS(R,\underline{t})=N(t_1,\ldots, t_r)/\left((1-t_1)^{s_1+1} \cdots (1-t_r)^{s_r+1}\right)$. 

\begin{theorem}\label{thm1}
	Write the Hilbert polynomial of $R$ as
	\begin{equation}\label{eq11}
		P(\underline{u}, R) = \sum_{\alpha=\underline{0}}^{\underline{s}} c_{\alpha}\binom{u_1+\alpha_1}{\alpha_1} \cdots \binom{u_r+\alpha_r}{\alpha_r}.
	\end{equation}
	Then
	\[c_{\alpha}=\frac{(-1)^{|\underline{s}-\alpha|}}{(s_1-\alpha_1)! \cdots (s_r-\alpha_r)!} \cdot \frac{\partial^{|\underline{s}-\alpha|}N}{\partial t_1^{s_1-\alpha_1} \cdots \partial t_r^{s_r-\alpha_r}}\scalebox{2.5}{\ensuremath \mid}_{(t_1, \ldots, t_r)=\underline{1}}.\]
\end{theorem}

Note that 
\[\binom{u_i+\alpha_i}{\alpha_i}= \frac{1}{\alpha_i!}u_i^{\alpha_i}+\text{ lower degree terms}.\] 
So if we write $P(\underline{u})$ as in \eqref{eq11}, then $c_{\alpha}=e_{\alpha}$ for all $\alpha \in \N^{r+1}$ with $|\alpha|=d-1$.
Therefore, Theorem \ref{thm1} gives an expression for $e_{\alpha}.$

\begin{remark}\label{exp}
	Let $I'_0, I'_1,\ldots,I'_r$ denote the images of ideals $I_0, I_1,\ldots,I_r$ in the
	ring $A/(0 : I^\infty),$ where $I = I_1 \cdots I_r.$ Put $R' = R(I'_0 |I'_1,\ldots, I'_r).$ Then for $\underline{u}$ large, $P_R(\underline{u}) = P_{R'}(\underline{u})$ (see \cite[Theorem 1.2]{Trung-Verma} for details). Therefore, in case $\grade I_i=0$ for some $i$, the user needs to work in the quotient ring $A/(0:I^\infty)$ and input the images of the ideals in the quotient ring.
\end{remark}

\noindent
{\bf Algorithm.} The algorithm for the function \texttt{MixedMultiplicity} uses the above ideas to calculate the mixed multiplicity. Let $I_0,I_1,\ldots,I_r$ be a set of ideals of a Noetherian ring $R$ of dimension $d \geq 1$, where $I_0$ is primary to the maximal ideal and $\grade(I_i)>0$ for all $i$; $\underline{a}=(a_0,a_1,\ldots,a_r) \in \N^{r+1}$ with $|\underline{a}|=d-1.$ \\
{\bf Input}: The sequence $W=((I_0,I_1,\ldots,I_r),(a_0,a_1,\ldots,a_r))$.
\begin{enumerate}[1.]
	\item Compute the defining ideal of the multi-Rees algebra using the function \texttt{multiReesIdeal} and use it to find the Hilbert series of $R(I_0 \mid I_1,\ldots,I_r).$
	
	\item Extract the powers of $(1-T_i)$ in the denominator of the Hilbert series.
	
	\item Calculate $e_{\underline{a}}$ using the formula given in Theorem \ref{thm1}.
\end{enumerate}
{\bf Output}: The mixed multiplicity $e_{\underline{a}}(I_0 \mid I_1,\ldots,I_r).$

\begin{example}
%	Let $R=\mathbb{Q}[x,y,z,w]$, $\m=(x,y,z,w)$, and $I=(xyw^3, x^2yw^2, xy^3w, xyz^3).$ We compute the mixed multiplicity $e_{(0,1,1,1)}(\m \mid I,I,I).$
	{\small
		\begin{verbatim}
			i1 : R = QQ[x,y,z,w];
			i2 : I = ideal(x*y*w^3,x^2*y*w^2,x*y^3*w,x*y*z^3);
			i3 : m = ideal vars R;
			i4 : mixedMultiplicity ((m,I,I,I),(0,1,1,1))
			o4 = 6
		\end{verbatim}
	}
\end{example}

In the following example we 
% explain how one can 
make a session in Macaulay2 to find the mixed multiplicity in the situation when some ideal has grade zero. In the following example, we use the fact that $(0: I^\infty)=(0: (I^t)^\infty)$ for each $t \geq 1$.
\begin{example}
	Let $S=\mathbb{Q}[x,y,z,w]/(xz,yz)$, $\m=(x,y,z,w)$, and $I=(x,y).$ Notice that $\grade I=0$, since $I \in \Ass S$.
	
	{ \small
		\begin{multicols}{2}
			\begin{verbatim}
				i1 : S = QQ[x,y,z,w]/ideal(x*z, y*z);
				i2 : I = ideal(x,y);
				i3 : m = ideal vars S;
				i4 : L = saturate(sub(ideal 0, S), I);
				i5 : T = S/L;
				i6 : J = substitute(I, T);n = substitute(m, T); 
				i8 : dim T
				o8 = 3
				i9 : mixedMultiplicity ((n,J,J,J),(1,0,1,0))
				o9 = 1
			\end{verbatim}
		\end{multicols}
	}	
	% i6 : o = sub(ideal 0, S);
	% i12 : MixedMultiplicity ((n,J,J),(0,1,1))
	% o12 = 0
	% In \texttt{i13} and 

\end{example}

To calculate mixed multiplicity, the function \texttt{MixedMultiplicity} computes the Hilbert polynomial of the graded ring $\bigoplus I_0^{u_0}I_1^{u_1} \cdots I_r^{u_r}/I_0^{u_0+1}I_1^{u_1} \cdots I_r^{u_r}$ . In particular, if $I_1,\ldots,I_r$ are also $\m$-primary ideals, then $e_{(a_0,a_1,\ldots,a_r)}(I_0 \mid I_1,\ldots,I_r) = e(I_0^{[a_0+1]}, I_1^{[a_1]}, \ldots, I_r^{[a_r]})$ (see \cite[Definition 17.4.3]{Huneke-Swanson} for the definition of $e(I_0^{[a_0+1]}, I_1^{[a_1]}, \ldots, I_r^{[a_r]})$). Therefore, to compute the $(a_0+1, a_1,\ldots, a_r)$-th
mixed multiplicity of $I_0,I_1,\ldots,I_r$, one needs to enter the sequence $(a_0,a_1,\ldots,a_r)$ in the function. The same is illustrated in the following example.
\begin{example}
%	Let $R=\mathbb{Q}[x,y,z]$ and $\m$ be the maximal homogeneous ideal of $R.$ Let $f=z^5 + xy^7 + x^{15}$ and $J(f) = (y^7+15x^{14}, xy^6, z^4)$ be the Jacobian ideal of $f.$ We calculate the mixed multiplicities $e(\m) = e(\m^{[3]}, J(f)^{[0]})$ and $e(\m^{[2]},J(f)^{[1]}).$
	{\small
		\begin{multicols}{2}
			\begin{verbatim}
				i1 : R = QQ[x,y,z];
				i2 : m = ideal vars R;
				i3 : f = z^5 + x*y^7 + x^15;
				i4 : I = ideal(apply(0..2, i -> diff(R_i,f)));
				    i5 : mixedMultiplicity ((m,I),(2,0))
				    o5 = 1
				    i6 : mixedMultiplicity ((m,I),(1,1))
				    o6 = 4
			\end{verbatim}
		\end{multicols}
	}
\end{example}

\section{Mixed volume of lattice polytopes}\label{SecMixedVol}

The Minkowski sum of two polytopes $P$ and $Q$ in $\mathbb{R}^n$ is defined as the polytope $P + Q = \{a + b \mid a \in P, b \in Q \}.$ The $n$-dimensional mixed volume of a collection of $n$ polytopes $Q_1,\ldots,Q_n$ in $\mathbb{R}^n$, denoted by $MV_n(Q_1,\ldots, Q_n)$, is the coefficient of $\lambda_1 \cdots \lambda_n$ in $\vol_n(\lambda_1Q_1 + \cdots + \lambda_n Q_n).$ Given a collection of lattice polytopes in $\mathbb{R}^n$, Trung and Verma  proved that their mixed volume is equal to a mixed multiplicity of a set of homogeneous ideals. 
\begin{corollary}[{\rm {\cite[Corollary 2.5]{Trung-Verma}}}] \label{cor1-JKV}
Let $Q_1,\ldots,Q_n$ be an arbitrary collection of lattice polytopes in $\mathbb{R}^n$. Let $R = k[x_0, x_1,\ldots, x_n]$ and $\m$ be the maximal graded ideal of $R$. Let $M_i$ be any set of monomials of the same degree in $R$ such that $Q_i$ is the convex hull of the lattice points of their dehomogenized monomials in $k[x_1,\ldots,x_n]$. Let $I_j$ be the ideal of $R$ generated by the monomials of $M_j$. Then 
$ MV_n(Q_1,\ldots,Q_n) = e_{(0,1,...,1)}(\m \mid I_1,\ldots, I_n).$
\end{corollary}	

We use this result to construct an algorithm which calculates the mixed volume of a collection of lattice polytopes. We also give an algorithm which outputs the homogeneous ideal corresponding to the vertices of a lattice polytope. 

Let $Q$ be a lattice polytope in $\mathbb{R}^n$ with the set of vertices $\{p_1, \ldots, p_r\} \subseteq \N^n$. We first compute the corresponding homogeneous ideal $I$ in the ring $R = k[x_1,\ldots, x_{n+1}].$ We write a function \texttt{homIdealPolytope} which requires as an input the list $W = \{p_1,p_2,\ldots, p_r\}$ and produces as an output the homogeneous ideal corresponding to the lattice points of $Q.$

We write a function \texttt{mMixedVolume} to calculate the mixed volume of a collection of $n$ lattice polytopes in $\mathbb{R}^n.$ Let $Q_1,\ldots,Q_n$ be an arbitrary collection of lattice polytopes in $\mathbb{R}^n$. Let $R=k[x_1,\ldots,x_{n+1}]$ and let $I_i$ be the homogeneous ideal of $R$ such that the polytope $Q_i$ is the convex hull of the lattice points of the dehomogenization of a set of monomials that generates $I_i$ in $k[x_1,\ldots,x_n]$, for all $i$. Each of these homogeneous ideals can be obtained by giving the lattice points of each polytope as an input in the function \texttt{homIdealPolytope}. The function \texttt{mMixedVolume} takes the list $\{I_1,\ldots, I_n\}$ as an input and produces the mixed volume of $Q_1, \ldots, Q_n$ as an output. The function can also take the list of lists of vertices of the polytope as an input to compute their mixed volume. Since calculating the mixed volume is same as calculating a mixed multiplicity, the algorithm of the function \texttt{mMixedVolume} is similar to the algorithm of the function \texttt{MixedMultiplicity}.

\begin{example} \label{ex3}
	We calculate the mixed volume of a {\it cross polytope}. 
	An $n$-cross polytope ($\beta_n$) is the convex hull of the points formed by permuting the coordinates $(\pm 1, 0, \ldots,0) \in \mathbb{R}^n.$
	\begin{align*}
	\beta_n 
	&= \{ (x_1,\ldots,x_n) \in \mathbb{R}^n \mid |x_1|+\cdots +|x_n| \leq 1 \} \\
	&= \text{conv }\{(\pm 1,0,\ldots,0), (0,\pm 1,0,\ldots,0), \ldots, (0,0,\ldots,0,\pm 1) \}.
	\end{align*}  
The volume of an $n$-cross polytope is $2^n/n!$ (\cite[Theorem 2.1]{crosspolytope}) and hence the mixed volume is $2^n.$ We say that a polytope is a \emph{$(0,1)$-polytope} if its each vertex coordinates are $0$ or $1$, that is, whose vertex set is a subset of $\{0,1\}^d$ of the unit cube. In this example, we calculate mixed volume of a $2$-cross polytope and a $2$-dimensional $(0,1)$-polytope.
%     A polytope P such that each coordinate of every vertex of P is either 0 or 1
	{ \small
	\begin{multicols}{2}
	\begin{verbatim}
		i1 : A = {(0,1),(1,0),(0,-1),(-1,0)};
		i2 : mMixedVolume {A,A}
		o2 = 4
		i3 : I = homIdealPolytope A;
		i4 : B = {(0,0),(0,1),(1,0),(1,1)};
		i5 : J = homIdealPolytope B;
		i6 : mMixedVolume {I, sub(J, vars ring I)}
		o6 = 4
		\end{verbatim}
	\end{multicols}
}
% 		i3 : mMixedVolume {I,I}
% 		o3 = 4
\end{example}

The proposed function \texttt{mMixedVolume} takes less time to compute the mixed volume of a $3$-cross polytope than the existing function \texttt{mixedVolume} in the \texttt{Polyhedra} package.
{ \small
\begin{multicols}{2}
\begin{verbatim}
i1 : needsPackage "Polyhedra"; 
i2 : Q = crossPolytope 3;
i3 : time mixedVolume {Q,Q,Q};
     -- used 238.277 seconds
------- n-cross polytope 
i4 : CP = n -> ( 
          U = (i,p) -> (1..n)/(j -> if 
          j == i then p else 0);           
        flatten toList apply(1..n, i -> 
        toList(U(i,1), U(i,-1))) 
        );	
i5 : time mMixedVolume {CP(3), CP(3), CP(3)}
     -- used 3.71303 seconds
o5 = 8
		\end{verbatim}
	\end{multicols}
}

\section{Sectional Milnor numbers}\label{SecMilNo}

In this section, we give an algorithm to compute the sectional Milnor numbers. We use Teissier's observation of identifying the sectional Milnor numbers with mixed multiplicities to achieve this task. In \cite{teissier1973}, Teissier conjectured that invariance of the Milnor number implies invariance of the sectional Milnor numbers. The conjecture was disproved by Jo\"el Brian\c{c}on and Jean-Paul Speder. We verify their example using our algorithm.
% and also by explicitly calculating the mixed multiplicities.

Suppose that the origin is an isolated singular point of a complex analytic hypersurface $H=V(f)\subset \mathbb{C}^{n+1}.$ Let $f_{z_i}$ denote the partial derivative of $f$ with respect to $z_i.$ Set
\[\mu=\dim_{\mathbb{C}}\frac{\mathbb{C}\{z_0, z_1, \dots, z_n\}}{(f_{z_0}, f_{z_1},\dots, f_{z_n})}.\]
The number $\mu$ is called the {\it Milnor number} of the hypersurface $H$ at the origin. Teissier, in his Carg\`ese paper \cite{teissier1973}, refined the notion of Milnor number by replacing it with a sequence of Milnor numbers of intersections with general linear subspaces. Let $(X, x)$ be a germ of a hypersurface in  $\mathbb{C}^{n+1}$ with an isolated singularity. The Milnor number of $X\cap E$, where $E$ is a general linear subspace of dimension $i$ passing through $x$, is called the {\it $i^{th}$-sectional Milnor number} of $X.$ It is denoted by $\mu^{(i)}(X, x).$

Let $R=\mathbb{C}[x_1,\ldots,x_n]$ be a polynomial ring in $n$ variables, $\mathfrak{m}$ be the maximal graded ideal and $f \in R$ be any polynomial. Let $J(f) = (f_{x_1},\ldots,f_{x_n})$ be the Jacobian ideal. In 1973, Teissier proved that the $i^{th}$-mixed multiplicity, $e(\m^{[n-i]}, J(f)^{[i]})$, is equal to the $i^{th}$-sectional Milnor number of the singularity. 

Using Theorem \ref{thm1}, one can now calculate the first $n-1$ mixed multiplicities of $\m$ and $J(f)$. We use the ideas in the previous section to write a function \texttt{secMilnorNumbers} for computing the sectional Milnor numbers. With a polynomial $f$ given as an input, the algorithm calculates the Jacobian ideal of $f$ and then using the function \texttt{multiReesIdeal}, it finds the defining ideal of $\R(\m, J(f))$. This helps to find the Hilbert series of the special fiber $\f(\m,J(f))= \R(\m, J(f)) \otimes_R R/\mathfrak{m}$. Using the formula given in Theorem \ref{thm1}, it then calculates the mixed multiplicities.

\begin{example}
%Let $R=\mathbb{Q}[x,y,z]$ and $f=x^2y+y^2z+z^3$ be a polynomial in the ring. In order to calculate the sectional Milnor numbers, we make the following session in Macaulay2.
{\small
\begin{multicols}{2}
	\begin{verbatim}
		i1 : R = QQ[x,y,z];
		i2 : f = x^2*y+y^2*z+z^3;
		i3 : secMilnorNumbers(f)
		
		o3 = HashTable{0 => 1}
		               1 => 2
		               2 => 4
		o3 : HashTable
	\end{verbatim}
	\end{multicols}
}
\end{example}

%\subsection{Verifying example of Jo\"el Brian\c{c}on and Jean-Paul Speder}

%Teissier in \cite{teissier1973} conjectured that invariance of the Milnor number implies invariance of the sectional Milnor numbers. In \cite{Briancon-Speder}, Jo{\"e}l Brian\c{c}on and Jean-Paul Speder disproved the conjecture by giving a counter-example. 
Jo{\"e}l Brian\c{c}on and Jean-Paul Speder  (in \cite{Briancon-Speder}) considered the family of hypersurfaces ${\bf X}_{t} \in \C^3$ defined by 
$F_t(x,y,z)=z^5 + ty^6z + xy^7 + x^{15}=0$. They proved that the topological type of ${\bf X}_t$ is constant whereas the topological type of the section of ${\bf X}_t$ by a general plane varies. One can verify the example using the methods discussed above. For instance, consider the ideals $\m=(x,y,z)$ and $J(F_t)=(\partial F_t/\partial x, \partial F_t/\partial y, \partial F_t/\partial z)$ in the ring $\mathbb{C}[x,y,z]$.
In the \href{https://arxiv.org/pdf/1902.07384.pdf}{arXiv} version of this article, we show that while $e(J(F_t))$ is independent of $t$ but  $e(\m^{[1]}, J(F_t)^{[2]})$ depends on $t$. 
% This thereby verifying 
The following Macaulay2 session demostrates the example given by Brian\c{c}on and Speder. 
%The following displays the working in Macaulay2.
{\small
\begin{multicols}{2}
\begin{verbatim}
i1 : QQ[t];
i2 : k = frac oo;
i3 : R = k[x,y,z];
i4 : f = z^5 + t*y^6*z + x*y^7 + x^15;
i5 : secMilnorNumbers (f)
o5 = HashTable{0 => 1 }
               1 => 4
               2 => 26
i6 : g = z^5 + x*y^7 + x^15;
i7 : secMilnorNumbers (g)
o7 = HashTable{0 => 1 }
               1 => 4
               2 => 28
\end{verbatim}
\end{multicols}
}

\bibliographystyle{plain}
\bibliography{Goel-Mukundan-Roy-Verma-2021}
\end{document}